\documentclass[12pt, leqno]{amsart}

\usepackage[OT2,T1]{fontenc}
\usepackage{indentfirst}
\usepackage{amstext}
\usepackage{amsthm}
\usepackage{amsopn}
\usepackage{amsfonts}
\usepackage{amsmath}
\usepackage{amssymb}
\usepackage{latexsym}
\usepackage[english]{babel}
\usepackage{amscd}

\usepackage{amsmath}
\usepackage[all,cmtip]{xy}
\usepackage{leftidx}
\usepackage{graphicx}
\usepackage{xcolor}
\usepackage[shortlabels]{enumitem}
\usepackage{etoolbox}
\usepackage{hyperref}
\patchcmd{\section}{\normalfont\scshape\centering}{\normalfont\bfseries}{}{}
\patchcmd{\subsection}{-.5em}{.5em}{}{}

\usepackage{xpatch}

\xpatchcmd{\proof}
  {\itshape}
  {\bfseries}
  {}
  {}

\newtheorem{theo}{{Theorem}}[section]
\newtheorem{coro}[theo]{{Corollary}}
\newtheorem{lemma}[theo]{{Lemma}}
\newtheorem{prop}[theo]{Proposition}

\theoremstyle{definition}
\newtheorem{remark}[theo]{\textbf{Remark}}
\newtheorem{defn}[theo]{Definition}
\newtheorem{example}[theo]{Example}
\numberwithin{equation}{section}

\newtheorem{notation}[theo]{Notation}

\newcommand{\ZZ}{\mathbf{Z}}

\newcommand{\id}{\mathrm{id}}

\DeclareFontEncoding{OT2}{}{} 

\begin{document}
\tolerance 400 \pretolerance 200 \selectlanguage{english}

\title[Dynamics of elliptic hyperk\"ahler manifolds]{Dynamical degrees of automorphisms of non-projective hyperk\"ahler manifolds}
\author{Eva  Bayer-Fluckiger, Stevell Muller and Matthias Schütt}

\AtEndDocument{\bigskip{\footnotesize%
  \textsc{Eva Bayer-Fluckiger\newline EPFL-FSB-MATH, Station 8, 1015 Lausanne, Switzerland}\newline  
  \textit{E-mail address}: \texttt{eva.bayer@epfl.ch} \par
  \addvspace{\medskipamount}
 \textsc{Stevell Muller\newline Institut f\"ur Algebraische Geometrie, Leibniz Universit\"at
  Hannover, Welfengarten 1, 30167 Hannover, Germany} \newline
  \textit{E-mail address}: \texttt{muller@math.uni-hannover.de} \par
  \addvspace{\medskipamount}
  \textsc{Matthias Schütt\newline Institut f\"ur Algebraische Geometrie, Leibniz Universit\"at
  Hannover, Welfengarten 1, 30167 Hannover, Germany\newline
  Riemann Center for Geometry and Physics, Leibniz Universit\"at
  Hannover, Appelstrasse 2, 30167 Hannover, Germany}\newline
  \textit{E-mail address}: \texttt{schuett@math.uni-hannover.de}
}}

\date{\today}

\begin{abstract} 
We determine the dynamical degrees of bimeromorphic automorphisms of hyperk\"ahler manifolds of all known deformation types with irreducible characteristic polynomial. 
\end{abstract}
\maketitle

\small{} \normalsize

\medskip

\selectlanguage{english}
\section{Introduction}

A compact K\"ahler manifold $X$ is called \emph{hyperk\"ahler} if it is simply-connected and if $H^0(X, \Omega^2_X)$ is 1-dimensional spanned by a nowhere degenerate holomorphic 2-form. By definition, hyperk\"ahler manifolds have even complex dimension; 2-dimensional hyperk\"ahler manifolds are K3 surfaces.

To this date, every known example of a hyperk\"ahler manifold is deformation equivalent to a (crepant) resolution of the Albanese fiber of a moduli space of semistable sheaves on a projective $K$-trivial surface (i.e., K3 or abelian surface). We use the following common notation for the known deformation types of hyperk\"ahler manifolds of dimension greater than 2:

\begin{enumerate}
    \item $\mathrm{K3}^{[n]}$ ($n\geq 2$) for the deformation classes of the Hilbert schemes of $n$ points on K3 surfaces \cite[Th\'eor\`eme 3]{bea83},
    \item $\mathrm{Kum}_n$ ($n\geq 2$) for the deformation classes of the Albanese fibers of the Hilbert schemes of $n+1$ points on abelian surfaces \cite[Th\'eor\`eme~4]{bea83},
    \item OG10 for the deformation class of O'Grady's 10-dimensional examples \cite{og99},
    \item OG6 for the deformation class of O'Grady's 6-dimensional examples \cite{og03}.
\end{enumerate}

The second integral cohomology $H^2(X, \mathbf{Z})$ of a hyperk\"ahler manifold $X$ comes equipped with a nondegenerate integral quadratic form $q_X$, known as \emph{Beauville--Bogomolov--Fujiki} (BBF) quadratic form \cite[Proposition 5]{bea83}. The isometry class of the lattice $(H^2(X,\mathbf{Z}), q_X)$ is invariant in a fixed deformation type. It has been explicitly determined for the known hyperk\"ahler manifolds of dimension greater than 2 by \cite[Proposition 6 and 8]{bea83}, \cite[Corollary~3.5.12]{rap07} and \cite[Theorem~4.3]{rap08} (cf.\ Table~\ref{tab:info hk} for more details).

\begin{table}
    \caption{Topological invariants of the known hyperk\"ahler manifolds}\label{tab:info hk}
    
\centering\setlength{\tabcolsep}{6pt}
\renewcommand\arraystretch{1.5}
\resizebox{!}{!}{
\begin{tabular}{ccccc}
        deformation type of $X$&$\operatorname{dim}X$&$b_2(X)$&$\operatorname{sign}q_X$&$(H^2(X, \mathbf{Z}), q_X)$\\

         \hline
        
        $\mathrm{Kum}_n$&$2n$&7&$(3,4)$&$U^{\oplus3}\oplus A_1(n+1)$\\
         \hline
         
         OG6&6&8&$(3,5)$&$U^{\oplus3}\oplus A_1^{\oplus2}$\\
        \hline
        K3&2&22&$(3,19)$&$U^{\oplus3}\oplus E_8^{\oplus2}$\\
        \hline
        $\mathrm{K3}^{[n]}$&$2n$&23&$(3, 20)$&$U^{\oplus3}\oplus E_8^{\oplus2}\oplus A_1(n-1)$\\
         \hline
        
        OG10&10&24&$(3,21)$&$U^{\oplus3}\oplus E_8^{\oplus2}\oplus A_2$\\
        \hline
         
    \end{tabular}}
    
\end{table}
\medskip

Let $X$ be a hyperk\"ahler manifold and let $\operatorname{NS}(X) = H^2(X,\mathbf{Z})\cap H^{1,1}(X,\mathbf{R})$ be its N\'eron--Severi lattice. It is isomorphic to $\operatorname{Pic}(X)$ via the first Chern map $c_1(-)$, and so it is of rank $\rho(X)$, the Picard rank of $X$. We distinguish three cases:
\begin{enumerate}
    \item either $\operatorname{NS}(X)$ is degenerate and we say that $X$ is \emph{parabolic},
    \item or $\operatorname{NS}(X)$ is nondegenerate of signature $(1, \rho(X)-1)$ and we say that $X$ is \emph{hyperbolic},
    \item or $\operatorname{NS}(X)$ is negative definite and we say that $X$ is \emph{elliptic}.
\end{enumerate}
Note that $X$ is projective if and only if it is hyperbolic \cite[Theorem 3.11]{Huy}. 
In this work, we focus on the case where $X$ is elliptic.

\medskip

Given a hyperk\"ahler manifold $X$, we let $\operatorname{Bir}(X)$ denote the group of bimeromorphisms of $X$, i.e. meromorphic maps $X\dashrightarrow X$ admitting a meromorphic inverse. We let moreover $\operatorname{Aut}(X)\subset\operatorname{Bir}(X)$ be the subgroup of biholomorphisms of $X$ which are by definition the biregular bimeromorphisms of $X$. 

There is a natural orthogonal representation
\[
\operatorname{Bir}(X)\to O(H^2(X, \mathbf{Z}), q_X),\;\;\; f\mapsto f^\ast
\]
which has finite kernel (cf.\ Theorem \ref{thm:Hodge-Torelli}). 
Given a map $f\in \operatorname{Bir}(X)$, we define the \emph{dynamical degree} $\lambda(f)$ of $f$ to be\
\[
\lambda(f) := \max\{|\lambda|\mid \lambda\in\mathbf{C}\text{ is an eigenvalue of }f^\ast_{\mathbf{C}}\colon H^2(X,\mathbf{C})\to H^2(X,\mathbf{C})\}.
\]
When $X$ is elliptic, it is known that the dynamical degree $\lambda(f)$ is either $1$ or a Salem number \cite[Theorem 3.4]{oguiso08}, the latter being defined as follows:

\begin{defn} A {\it Salem polynomial} is a monic, irreducible polynomial $S \in \mathbf{Z}[x]$ of degree greater than 2 such that 
$S(x) = x^{{\rm deg}(S)} S(x^{-1})$ holds and such that $S$ has exactly two roots off the unit circle, both positive
real numbers. The unique real root greater than $1$ of a Salem polynomial is called a {\it Salem number}.
\end{defn}

There are ample results on the dynamics on
K3 surfaces, starting from  \cite{GM} and \cite{Mc1} through \cite{Mc2}, \cite{Mc3} and  \cite{BT}, \cite{B 25}, \cite{YT},
from general to concrete, but to this day only a few for hyperk\"ahler manifolds (like the aforementioned general results in \cite{oguiso08} or the recent work \cite{doo26}).
The main incentive to push for concrete results on hyperk\"ahler manifolds, and specifically on elliptic ones, 
came from \cite{BGSb} and its results on the Hodge endomorphism algebras of elliptic hyperk\"ahler manifolds.

\medskip

With this terminology, our two main results are  the following.

\begin{theo} 
\label{thm1}
Let $S$ be a Salem polynomial of degree $24$ and let $\lambda$ be the corresponding Salem number. There exists an elliptic hyperk\"ahler manifold of {\rm OG10}-type having a biholomorphism 
of dynamical degree $\lambda$ if and only if $|3S(1)S(-1)|$ is a square, and either $|S(1)|$ is a square, or $|S(-1)|$ is a square.

\end{theo}

\smallskip

\begin{theo} 
\label{thm2}
Let $S$ be a Salem polynomial of degree $8$, and let $\lambda$ be the corresponding Salem number. There exists an elliptic hyperk\"ahler manifold of {\rm OG6}-type having a biholomorphism 
of dynamical degree $\lambda$ if and only if $|S(1)|$ and $|S(-1)|$ are both squares, or both twice a square.
\end{theo}

Combined with the results from \cite[Part~III]{B 25},
we derive the following generalization 
of the irreducible characteristic polynomial statement from the introduction:

\begin{coro}
\label{cor:char}
    Let $X$ be a hyperk\"ahler manifold and assume that $X$ admits a bimeromorphism $f$ such that $f^\ast$ is nontrivial with irreducible minimal polynomial, denoted by $S$.
    Then $X$ is elliptic with even $b_2(X)$, $f$ is biregular, and $S$ is a Salem polynomial of degree $b_2(X)$. 

    If moreover $X$ has one of the known deformation types (cf.\ Table \ref{tab:info hk}),
    then one of the following alternatives holds true:
    \begin{enumerate}
        \item $X$ is a K3 surface, 
        and
        $|S(1)|$ and $|S(-1)|$ are both squares.
        \item 
        $X$ is of OG6-type,
        and
        $|S(1)|$ and  $|S(-1)|$ are both squares or both twice a square.
        \item 
        $X$ is of OG10-type,
        and
        $|3S(1)S(-1)|$ as well as one of $|S(1)|$ and $|S(-1)|$ is a square.
    \end{enumerate}
\end{coro}

The analogous problem for reducible minimal polynomials is analysed 
to some extent in Section \ref{2toN}
(see Theorems \ref{thm:k3 to hidger dim} and \ref{thm_from_tori}).

\medskip

The paper is organized as follows. We first gather some definitions and preliminaries about lattices, isometries of quadratic forms and Witt groups. We then prove the lattice-theoretic counterparts of Theorems~\ref{thm1} and \ref{thm2}. 
After recalling the basics regarding the period maps and the Hodge-theoretic version of the Torelli theorems for hyperk\"ahler manifolds, we use the  lattice-theoretic existence results to complete the proofs of the main theorems. 
The paper concludes by relating the dynamical degrees of many other bimeromorphisms of hyperk\"ahler manifolds with the ones of complex K3 surfaces and complex 2-tori.

\section{Definitions, notation and basic facts}\label{basic}

We start by recalling some notions and results from \cite {M}, \cite {B 15}, \cite {BT} and \cite {B 24}.
Let $K$ be a field.

\subsection{Isometries}

Let $V$ be a finite dimensional
$K$-vector space, and let $q\colon V \times V \to K$ be a nondegenerate symmetric bilinear form. We refer to $(V, q)$ as a symmetric $K$-bilinear form. An {\it isometry} of $(V,q)$ is by definition
an isomorphism $\varphi\colon V \to V$ such that $q(\varphi x,\varphi y) = q(x,y)$ for all $x, y \in V$.

\medskip

If
$g \in K[x]$ is a monic polynomial such that $g(0) \not = 0$,  set $$g^*(x) := {g(0)}^{-1} x^{{\rm deg}(g)} g(x^{-1});$$ we say that $g$ is
{\it symmetric} if $g^* = g$.  
It is well-known that the characteristic polynomial of an isometry
is symmetric (see for instance \cite[Proposition~1.1]{B 15}). Moreover the following holds:

\begin{lemma}[{{\cite[Corollary~5.2]{B 15}}}]
    Let $(V, q)$ be a symmetric $K$-bilinear form and let $g\in K[x]$ be the characteristic polynomial of an isometry of $(V, q)$. Suppose that $g(1)g(-1)\neq 0$. Then
    \begin{equation}\label{eq:det formula}
        \det(q) = g(1)g(-1)(K^\times)^2 \in K^\times/(K^\times)^2.
    \end{equation}
\end{lemma}
For the rest of the paper, we refer to (\ref{eq:det formula}) as \emph{determinant formula}.

\subsection{Equivariant Witt groups and residue maps}

 Let $\Gamma$ be the infinite cyclic group denoted multiplicatively, and
let $K[\Gamma]$ be the associated group ring. Sending $\gamma \in \Gamma$ to $\gamma^{-1}$
induces a $K$-linear involution of $K[\Gamma]$; let $W_{\Gamma}(K)$ be the  Witt
group of symmetric $K[\Gamma]$-bilinear forms (see  \cite[Section~3]{BT}). If $K$ is the field of fractions of some Dedekind domain, we let $W^b_{\Gamma}(K)$ be the subgroup of $W_\Gamma(K)$ generated by bounded forms \cite[Definition~4.2]{BT}.

\medskip

Let $(V, q)$ be a symmetric $K$-bilinear form, and let $\varphi$ be an isometry of $(V, q)$. Let $\gamma$ be a generator
of $\Gamma$. Let us define a $K[\Gamma]$-module structure on $V$ by setting 
$\gamma\cdot v := \varphi(v)$ for all $v \in V$. This gives $(V, q)$ the structure of a symmetric $K[\Gamma]$-bilinear form, and
we obtain an element $[V] = [V,q]$ of $W_{\Gamma}(K)$.

\medskip 

If $K$ is a local field with residue field $\kappa$, we have a residue map $\partial\colon W^b_{\Gamma}(K) \to
W_{\Gamma}(\kappa)$, see \cite[Section~4]{BT}. A symmetric $K[\Gamma]$-bilinear form contains a unimodular
lattice if and only if it is defined on a bounded form, and its class in $W^b_{\Gamma}(K)$ lies in the kernel of $\partial$
(cf. \cite[Theorem~4.3(iv)]{BT}). 
 
 \subsection{Symmetric polynomials over the integers}

 The following definitions and notation are used throughout the paper.

\begin{defn} Let $g \in \mathbf{Z}[x]$ be a monic, symmetric polynomial. We say that $g$ is of 

\begin{itemize}
\item
{\it type} 0 if $g$ is a product of powers of $x-1$ and of $x+1$;

\item
{\it type} 1  if $g$ is a product of powers of monic, symmetric, irreducible polynomials in
$\mathbf{Z}[x]$ of even degree;

\item
{\it type} 2 if $g$ is a product of polynomials of the form $h h^*$, where $h \in \mathbf{Z}[x]$ is
monic, irreducible, and $h \not = \pm  h^*$. 
\end{itemize}
\end{defn}

It is well-known that every monic, symmetric polynomial  is a product of polynomials
of type {\rm 0, 1} and {\rm 2}.

\begin{example}
    Examples of irreducible polynomials of type 1 are cyclotomic and Salem polynomials.
\end{example}

\begin{notation}\label{notat:even symmetric poly decomp}
    If $g \in \mathbf{Z}[x]$ is a monic, symmetric polynomial, we write $g = g_0 g_1 g_2$, where $g_i$ is the product of the irreducible factors of type $i$ of $g$. We have
    $g_0(x) = (x-1)^{n^+(g)} (x+1)^{n^-(g)}$ for some integers $n^+(g),n^-(g) \geqslant 0$, and $g_1 = \underset{h \in I_1(g)} \prod h^{n_h(g)}$, where $I_1(g)$ is the set of irreducible factors of type 1 of $g$.
\end{notation}

\subsection{Lattices and discriminant groups}

Let $R = \mathbf{Z}$ or $\mathbf{Z}_p$ for some $p$ prime, and let $K := \operatorname{Frac}(R)$. We call $R$-lattice a finitely generated free $R$-module equipped with a $K$-valued nondegenerate symmetric bilinear form. If the form takes values in $R$, we call the $R$-lattice integral. In this paper, ADE root lattices associated to the corresponding Dynkin diagrams are assumed to be negative definite.

\begin{notation}
    Let $n\geq 1$ be a positive integer, let $m\in\ZZ\setminus\{0\}$, and let $B$ be a group. Given a symbol $D$ representing the data consisting of
    \begin{itemize}
        \item a group $A$,
        \item a $B$-valued $n$-ary function $q$ on $A$,
    \end{itemize}
    we let $D(m)$ represent the same group $A$ but equipped with the function $mq$.
\end{notation}

 Given an integral $R$-lattice $L$, we let $A(L) := L^\vee/L$ be its discriminant group; it is a finite abelian group. We say that $L$ is $n$-elementary, for some integer $n\geq 1$, if $nL^\vee\subset L$. When the quadratic form on $L$ takes values in $2R$ (i.e., $L$ is even), the group $A(L)$ is equipped with a \emph{torsion quadratic form}
\[q_L\colon A(L)\to K/2R,\; x+L\mapsto (x,x)+2R.\]
Given a morphism $\varphi\colon L_1\to L_2$ of $R$-lattices, we denote by $A(\varphi)$ the induced morphism between discriminant groups; if $L_1$ and $L_2$ are even, $A(\varphi)$ respects the torsion quadratic forms. We denote by $\mathcal{M}(L, \varphi)$ the set of $\varphi$-stable even overlattices $L\subset M \subset L^\vee$ and by $\mathcal{H}(L, \varphi)$ the set of $A(\varphi)$-stable totally $q_L$-isotropic subgroups $H\subset A(L)$.

\begin{prop}[{{\cite[Propositions~1.4.1 and 1.4.2]{N}}}]\label{thm:corr even overlat and isotrop subgroup}
    Let $L$ be an even $R$-lattice and let $\varphi\in O(L)$ be an isometry. Then there is a one-to-one correspondence
    \[\begin{array}{ccc}
        \mathcal{M}(L, \varphi)&\leftrightarrow&\mathcal{H}(L, \varphi)\\
         M&\mapsto& M/L\\
         \{x\in L^\vee\ |\ x+L\in H\}&\mathrel{\reflectbox{\ensuremath{\mapsto}}}&H
    \end{array}.
   \]
   Moreover, given $H = M/L\in \mathcal{H}(L, \varphi)$, there is an isomorphism of torsion quadratic forms $A(M)\cong H^\perp/H$ which is compatible with the respective $\varphi$-actions, where the quadratic form on the right is induced from $q_L$.
\end{prop}

In particular, given an even $R$-lattice $L$ and an isometry $\varphi\in O(L)$, the discriminant groups of the $\varphi$-stable even overlattices of $L$ are completely determined by $A(L)$ and $A(\varphi)\in O(A(L),q_L)$. As a consequence, we deduce the following result, to which we give a short proof for the reader's convenience.

\begin{lemma}\label{lem:stable even overlattices}
    Let $L$ be an even unimodular $R$-lattice and let $\varphi\in O(L)$ be a semisimple isometry whose minimal polynomial is a product of two polynomials $g_1,\,g_2\in \mathbf{Z}[x]$ which are coprime in $\mathbf{Q}[x]$. For every $i\in\{1,2\}$, let $L_i := \ker g_i(\varphi)$ and let $\varphi_i:= \varphi\mid_{L_i}$. Then there is a canonical bijection
    \[\gamma\colon \mathcal{M}(L_1, \varphi_1)\to \mathcal{M}(L_2, \varphi_2)\]
    such that for all $M\in \mathcal{M}(L_1, \varphi_1)$, the following holds:
    \[A(\gamma(M))\cong A(M)(-1).\]
\end{lemma}

\begin{proof}
    Since $L$ is unimodular, the group $L/(L_1\oplus L_2)$ is the graph of an isomorphism of abelian groups
    \[a\colon A(L_1)\xrightarrow{\cong} A(L_2)\]
    such that both of the following hold:
    \begin{enumerate}
        \item for all $x\in A(L_1)$, we have $q_{L_2}(a(x)) = -q_{L_1}(x)$ (because $L$ is even),
        \item $A(\varphi_2)\circ a = a\circ A(\varphi_1)$ (because $\varphi_1\oplus \varphi_2$ preserves $L$).
    \end{enumerate}
    (See \cite[Chapter~1, Section~5]{N} for more details). In particular, the isomorphism $a$ induces a canonical bijection between $\mathcal{H}(L_1, \varphi_1)$ and $\mathcal{H}(L_2, \varphi_2)$ by sending $H$ to $a(H)$. The rest of the proof is now a consequence of Proposition~\ref{thm:corr even overlat and isotrop subgroup}.
\end{proof}

\section{Local results}\label{residuemaps}

Assume that $K$ is a non-archimedean local field of characteristic different from 2, that
$E$ is a finite field extension of $K$ with a nontrivial involution $\sigma \colon E \to E$ and
that $E_0$ is the fixed field of $\sigma$. Set $T(E,\sigma) := E_0^{\times}/{\rm N}_{E/E_0}(E^{\times})$. 
Let $\ell \colon E \to K$ be a nontrivial $K$-linear map such
that $\ell(\sigma(x) )= \ell(x)$ for all $x \in E$. Let $\mathcal{O}_K$ be the ring of integers of $K$, and let
$k$ be its residue field. Let $\mathcal{O}_E$ be the ring of integers of $E$, and let $\mathfrak{m}_E$ be its maximal ideal; set $\kappa_E := \mathcal{O}_E/\mathfrak{m}_E$. 

\medskip
If $n \geqslant 1$ is an integer, we define a map $$t_n\colon T(E,\sigma) \to W_\Gamma^b(K)$$ as follows. For
$[\lambda] \in T(E,\sigma)$, choose a representative $\lambda\in E_0^\times$ and let us denote by $h_{n,\lambda}$ the $n$-dimensional diagonal
hermitian form $\langle \lambda, 1, \dots, 1 \rangle$ over $E$ with respect to the involution $\sigma$, and set
$q_{n,\lambda} :=\ell(h_{n,\lambda})$. This construction turns $(E^n, q_{n, \lambda})$ into a symmetric $K$-bilinear form.  For every $\alpha\in E^\times$ with $\alpha\sigma(\alpha) = 1$, multiplication by $\alpha$ on $E^n$ defines an isometry of $(E^n, q_{n, \lambda})$. The triple $(E^n, q_{n,\lambda}, \alpha)$ defines a bounded symmetric $K[\Gamma]$-bilinear form; we denote by $[E, q_{n,\lambda},\alpha]\in W^b_{\Gamma}(K)$ the associated Witt class.
Recall that we have a homomorphism $\partial\colon W_{\Gamma}^b(K) \to W_{\Gamma}(k)$.

\begin{prop}[{{\cite[Proposition~11.1]{B 24}}}]\label{residue} For all integers $n \geqslant 1$, the following hold:

\medskip

 {\rm (a)} If $E/E_0$ is inert, then 

$$\partial  \circ t_n \colon T(E,\sigma) \to  W_{\Gamma}(k)$$ is a bijection.

\medskip

 {\rm (b)} If $E/E_0$ is ramified and ${\rm char}(k) \not = 2$, then 

$$\partial  \circ t_n \colon T(E,\sigma) \to W_{\Gamma}(k)$$ is injective, and its
image consists of the classes of forms of dimension $n[\kappa_E:k]\mod 2$.

\medskip
{\rm (c)} If $E/E_0$ is ramified and ${\rm char}(k) = 2$, then 

$$\partial  \circ t_n \colon T(E,\sigma) \to  W_{\Gamma}(k)$$ is constant.

\end{prop}

\medskip
If $p$ is a prime number, we denote by $v_p$ the $p$-adic valuation.

\begin{notation} If $g \in \mathbf{Z}[x]$ is a monic, irreducible, symmetric polynomial of even degree, set $E_g := \mathbf{Q}[x]/(g)$, let $\sigma_g \colon E_g \to E_g$ be the involution
induced by $x \mapsto x^{-1}$, and let $(E_g)_0$ be the fixed field of $\sigma_g$ in $E_g$. Let $\alpha \in E_g$ be the image of $x$.
\end{notation}

\begin{defn}\label{ramified} Let $g \in \mathbf{Z}[x]$ be a monic, irreducible, symmetric polynomial of even degree, and let $p$ be a prime number. We say that
$g$ is {\it ramified at} $p$ if there exists a place $w$ of $(E_g)_0$ above $p$ that is ramified in 
$E_g$; otherwise, we say that $g$ is {\it unramified} at $p$. We denote by $\Pi^r_g$ the set of prime numbers $p$ such that $g$ is ramified at $p$.

\medskip Let $p \in \Pi^r_g$ be an odd prime number.  If $w$ is a place of $(E_g)_0$ above $p$ that is ramified in $E_g$, we denote by $\kappa_w$  the
residue field of $w$, and by $\overline{\alpha}$ the image of $\alpha$ in $\kappa_w$; we denote by $\Omega_+$ the set of places $w$ above $p$ such that $\overline{\alpha} = 1$,
and by $\Omega_-$ the set of places $w$ above $p$ such that $\overline{\alpha} = - 1$. We denote by 
 $\Pi^{r,+}_g$ the set of prime numbers $p$ such that there exists a place $w$ above $p$ with $w \in \Omega_+$, and by $\Pi^{r,-}_g$ the set of 
 prime numbers $p$ such that there exists a place $w$ above $p$ with $w \in \Omega_-$.

\end{defn}

\begin{prop}[{{\cite[Proposition~2.9]{B 25}}}]\label{odd ramification} Let $g \in \mathbf{Z}[x]$ be a monic, irreducible, symmetric polynomial of even degree, and let $p$ be an odd prime number.
Then, both of the following hold:
\begin{enumerate}
    \item[{\rm (a)}] if $v_p(g(1))  \equiv 1 \mod 2$, then $p \in \Pi^{r,+}_g$,
    \item[{\rm (b)}] if $v_p(g(-1))  \equiv 1 \mod2$, then $p \in \Pi^{r,-}_g$.
\end{enumerate}

\end{prop}

The following two lemmas are direct consequences of \cite[Proposition~2.9]{B 25}.

\begin{lemma}\label{p odd}
    Let $p\neq 2$ be an odd prime number and let $g\in\mathbf{Z}[x]$ be a monic, irreducible, symmetric polynomial of even degree. If $v_p(g(1))$ and $v_p(g(-1))$ are both even, then there exists an even, unimodular $\mathbf{Z}_p$-lattice having a semisimple isometry with characteristic polynomial $g$.
\end{lemma}

\begin{lemma}\label{p even}
    Let $p = 2$. Let $g\in\mathbf{Z}[x]$ be a monic, irreducible, symmetric polynomial of even degree $2n\geq 2$. If $v_p(g(1))$ and $v_p(g(-1))$ are both even and if the class of $(-1)^ng(1)g(-1)$ in $\mathbf{Q}^\times/(\mathbf{Q}^\times)^2$ lies in $\{1, -3\}$, then there exists an even, unimodular $\mathbf{Z}_2$-lattice having a semisimple isometry with characteristic polynomial $g$.
\end{lemma}

Given a $\mathbf{Z}$-lattice $L$, we define the genus of $L$ to be the collection of all $\mathbf{Z}$-lattices $L'$ such that $L$ and $L'$ have the same signature and such that $L\otimes \mathbf{Z}_p\cong L'\otimes\mathbf{Z}_p$ for every prime number $p$.

\begin{prop}\label{thm:local to global irreducible polynomials}
    Let $L$ be an even $\mathbf{Z}$-lattice and let $g\in\mathbf{Z}[x]$ be a monic, irreducible, symmetric polynomial such that $L\otimes\mathbf{R}$ has an isometry with characteristic polynomial $g$. 
    Then the following are equivalent:
    \begin{enumerate}
        \item There exists a lattice $L'$ in the genus of $L$ admitting an isometry with characteristic polynomial $g$,
        \item For every prime number $p$, the $\mathbf{Z}_p$-lattice $L\otimes \mathbf{Z}_p$ admits an isometry with characteristic polynomial $g$.
    \end{enumerate}
\end{prop}

\begin{proof}
    The fact that (1) implies (2) follows from the definition of the genus of $L$, i.e. that $L'\otimes \mathbf{Z}_p\cong L\otimes \mathbf{Z}_p$ for every prime number $p$ if $L$ and $L'$ are in the same genus.

    To prove that (2) implies (1), we proceed as in the proof of \cite[Theorem 12.1]{B 25}. Since $g$ is irreducible, our assumptions together with \cite[Theorem 9.1]{B 15} imply that the quadratic space $L\otimes \mathbf{Q}$ has an isometry $\varphi$ with characteristic polynomial $g$ such that for every prime number $p$, the quadratic space $L\otimes \mathbf{Q}_p$ contains an even $(\varphi\otimes \id)$-stable $\mathbf{Z}_p$-lattice $L_p\cong L\otimes\mathbf{Z}_p$.
    Let $M\subset L\otimes \mathbf{Q}$ be a $\varphi$-stable $\mathbf{Z}$-lattice. Note that for all but finitely prime numbers $p$, the $(\varphi\otimes\id)$-stable $\mathbf{Z}_p$-lattices $L_p$ and $M\otimes \mathbf{Z}_p$ are isometric. Hence we may choose $L_p$ to be equal to $M\otimes\mathbf{Z}_p$, for all but finitely many prime numbers. Thus, according to \cite[81:14]{OM}, there exists a $\mathbf{Z}$-lattice $L'$ in $L\otimes \mathbf{Q}$ such that $L'\otimes\mathbf{Z}_p = L_p$ for \emph{all} prime numbers $p$. More precisely,
\begin{equation}\label{eq:explicit omeara}
    L'=\bigcap_{p\text{ prime}}\iota_p^{-1}(\iota_p(L\otimes\mathbf{Q})\cap L_p)\subset L\otimes\mathbf{Q},
\end{equation}
where for every prime number $p$, we let $\iota_p\colon L\otimes\mathbf{Q}\hookrightarrow L\otimes\mathbf{Q}_p$ be the natural inclusion. In particular, the lattice $L'$ is $\varphi$-stable by \eqref{eq:explicit omeara}. Moreover for all prime numbers $p$, the $\mathbf{Z}_p$-lattices $L'\otimes\mathbf{Z}_p$ and $L\otimes\mathbf{Z}_p$ are both isometric to $L_p$. Since $L\otimes \mathbf{Q} = L'\otimes\mathbf{Q}$ by construction, the previous statement is equivalent to saying that $L'$ lies in the genus of $L$. 

\end{proof}

\section{Isometries of lattices of type {\rm OG10}}

Throughout this section, we let $L := \Lambda_{\rm OG10} :=  U^{\oplus3}\oplus E_8^{\oplus2}\oplus A_2$, and we let $S$ be a Salem polynomial of degree 24. We shall prove the following:

\begin{theo}\label{thm:salem for og10} 
The lattice $L$ has an isometry $\varphi$ with characteristic polynomial $S$ if and only if $|3S(1)S(-1)|$ is a square, and either
$|S(1)|$ is a square, or $|S(-1)|$ is a square.
\end{theo}

\begin{remark}
    The proof of Theorem~\ref{thm:salem for og10} actually shows that if $\varphi\in O(L)$ has characteristic polynomial $S$ then $\varphi$ acts trivially on the discriminant group $A(L)$ of $L$ if and only if $|S(-1)|$ is a square (see the proof of Lemma~\ref{thm:necessary part og10} for more details).
\end{remark}

\begin{remark}\label{rem:unique up to iso - og10}
    According to \cite[Theorem 1.14.2]{N}, since the lattice $L$ is even indefinite of rank 24 and its discriminant group is cyclic, the lattice $L$ is unique in its genus, up to isometry. 
\end{remark}

We will need the following two auxiliary results to prove the two implications of Theorem \ref{thm:salem for og10}.

\begin{lemma}\label{thm:necessary part og10}
    If the lattice $L$ has an isometry with characteristic polynomial $S$, then $|3S(1)S(-1)|$ is a square, and one of $|S(1)|$ or $|S(-1)|$ is a square.
\end{lemma}

\begin{proof}
Denote the isometry by $\varphi$ and assume first that it induces a trivial action on the discriminant group $A(L)$ of $L$.
Consider an order 3 automorphism $\varphi'$ of $A_2(-1)$.
Then $(\varphi, \varphi')$ extends to an automorphism $\Phi\in O(\Lambda)$ 
for the even unimodular index 3 overlattice $\Lambda$ of $L\oplus A_2(-1)$ \cite[Corollary~1.5.2]{N}.
By construction, $\Phi$ has characteristic polynomial $S(x)(x^2+x+1)$,
so \cite[Theorem 6.1]{GM} applies to show that
\[
|3S(1)|, \;|S(-1)|, \textnormal{ and } {-3}S(1)S(-1)\]
are all squares.
If $\varphi$ induces a nontrivial action on $A(L)$,
then the same argument applies to the order 6 isometry $-\varphi'$ of $A_2(-1)$ to give that
\[
|S(1)|, \;|3S(-1)|, \textnormal{ and } {-3}S(1)S(-1)\]
are all squares.
Taken together, these two cases are equivalent to the statement of the lemma (modulo taking absolute values which follows from the inequalities $S(1) < 0 < S(-1)$).\qedhere

\end{proof}

\begin{prop}\label{thm:sufficient part og10 - local existence}
Assume that $|3S(1)S(-1)|$ is a square, and either
$|S(1)|$ is a square, or $|S(-1)|$ is a square. Then for every prime number $p$, the lattice $L\otimes\mathbf{Z}_p$ admits an isometry with characteristic polynomial $S$.
\end{prop}
\begin{proof}
    We split the proof depending on whether $p = 3$.\smallskip

    Let first $p\neq 3$ be a prime number. The integer 3 is invertible in $\mathbf{Z}_p$ and the hypothesis implies that both of $v_p(S(1))$ and $v_p(S(-1))$ are even. Moreover, since $-3S(1)S(-1)$ is a square and $\deg(S) = 24 = 2\cdot 12$, we have $(-1)^{12}S(1)S(-1)(\mathbf{Q}_2^\times)^2 = -3(\mathbf{Q}_2^\times)^2$. We apply Propositions~\ref{p odd} and \ref{p even} to assert that there exists an even unimodular $\mathbf{Z}_p$-lattice $M_p$ having an isometry with characteristic polynomial $S$. Since the $p$-valuations of $S(1)$ and $S(-1)$ are even, the determinant formula (\ref{eq:det formula}) implies that $\operatorname{det}M_p = -3(\mathbf{Z}_p^\times)^2 = \operatorname{det}(L\otimes\mathbf{Z}_p)$. 
        
    Since $p \not = 3$, the lattice $L\otimes\mathbf{Z}_p$ is even unimodular of determinant $-3(\mathbf{Z}_p^\times)^2$. Therefore we have $\operatorname{det}M_p =
  \operatorname{det}(L\otimes\mathbf{Z}_p)$, and both $M_p$ and $L\otimes\mathbf{Z}_p$  are even unimodular lattices over $\mathbf{Z}_p$. Since even unimodular $\mathbf{Z}_p$-lattices are uniquely determined, up to isometry, by their rank and determinant class (see for instance \cite[Proposition~1.9.1]{N}), we infer that $M_p\cong L\otimes \mathbf{Z}_p$ and we conclude.
  
    \smallskip
    
    Let now $p = 3$. We proceed as in the proof of Claim 1 of the proof of \cite[Proposition 2.6]{B 25}. For the rest of the proof, we fix some notation. Let $E := \mathbf{Q}[x]/(S)$, let $\alpha$ be the image of $x$ in $E$ and let ${E_0}$ be the real subfield of $E$, fixed by the $\mathbf{Q}$-linear involution $\sigma\colon \alpha\mapsto \alpha^{-1}$. According to Proposition~\ref{odd ramification}, the polynomial $S$ is ramified at 3. There is a decomposition
    \[E_0\otimes\mathbf{Q}_3 \cong \prod_{v\in \Omega_{sp}}E_{0,v}\times\prod_{v\in \Omega_{un}}E_{0,v}\times\prod_{v\in \Omega_{+}}E_{0,v}\times\prod_{v\in \Omega_{-}}E_{0,v}\]
    where $\Omega_{sp}$ and $\Omega_{un}$ denote the sets of 3-adic places of $E_0$ which split, resp. which are unramified, in $E$. The sets $\Omega_{\pm}$ denote the sets of 3-adic places $v$ of ${E_0}$ such that $E_{v}/E_{0,v}$ is ramified and the image $\overline{\alpha}$ of $\alpha$ in the residue field $\kappa_v$ of $E_v$ is $ \pm1$. Here $E_{0,v} := (E_0)_v$ and $E_v := E\otimes_{E_0}E_{0,v}$ are the associated completions. Let moreover $\Omega = \Omega_{sp}\cup\Omega_{un}\cup\Omega_+\cup\Omega_-$.
    
    For every $v\in\Omega_{sp}$, the group $T(E_v, \sigma\otimes \operatorname{id})$ is trivial and we set $\lambda_v = 1$; note that $\partial[E_v, q_{1,\lambda_v}, \alpha] = 0$ in $W_{\Gamma}(\mathbf{F}_3)$. Moreover, for every $v\in \Omega_{un}$, we  apply \cite[Proposition~6.4]{BT} to find an element $\lambda_v\in T(E_{0,v}, \sigma\otimes \operatorname{id})$ such that $\delta[E_v, q_{1,\lambda_v}, \alpha] = 0\in W_{\Gamma}(\mathbf{F}_3)$.
    
    Now, by hypothesis, either $v_3(S(1))$ or $v_3(S(-1))$ is odd. Suppose that $v_3(S(1))$ is odd; 
    the rest of the proof follows similarly in the case where $v_3(S(-1))$ is odd. Then Proposition~\ref{odd ramification} tells us that $\Omega_+$ is nonempty, and according to \cite[Lemma 6.8]{BT} we have that 
    \[\underset {v \in \Omega_+} \sum [\kappa_v : \mathbf{F}_3] \equiv \ v_3(S(1))\mod2 \]
   and  is odd by hypothesis. Hence there exists a place $v_0\in \Omega_+$ such that $[\kappa_{v_0}:\mathbf{F}_3]$ is odd. We  use Proposition~\ref{residue}(b) to show that the image of the map
    \[\partial\circ
    t_1\colon T(E_{v_0}, \sigma\otimes \operatorname{id})\to W_{\Gamma}(\mathbf{F}_3, +) = W(\mathbf{F}_3)
    \]
    contains the Witt class $\delta$ of a one-dimensional form over $\mathbf{F}_3$ of determinant $+1$. Here we denote by $W_{\Gamma}(\mathbf{F}_3, \pm)$ the subgroup of $W_\Gamma(\mathbf{F}_3)$ generated by the symmetric $\mathbf{F}_3[\Gamma]$-bilinear forms $(W, q)$ on which a generator of $\Gamma$ acts as $\pm\operatorname{id}$. We let $\lambda_{v_0}$ be a representative of a preimage of $\delta$.
    
    Finally, by our assumption and \cite[Lemma 6.8]{BT}, the following congruences hold:
    \[\underset {v \in \Omega_+\setminus\{v_0\}} \sum [\kappa_v : \mathbf{F}_3] \equiv \underset {v \in \Omega_-} \sum [\kappa_v : \mathbf{F}_3] \ \equiv \ 0\mod2.\]
    Therefore, according to \cite[Proposition~6.6]{BT}, for every $v\in \left(\Omega_+\setminus\{v_0\}\right)\cup \Omega_-$ we can find $\lambda_v\in T(E_v, \sigma\otimes \operatorname{id})$ such that 
    \[\sum_{v\in \left(\Omega_+\setminus\{v_0\}\right)\cup \Omega_-}\partial[E_v, q_{1,\lambda_v}, \alpha] = 0 \in W(\mathbf{F}_3)\]
    holds. Hence, the element
    \[\lambda := (\lambda_v)_{v\in\Omega}\in T(E\otimes\mathbf{Q}_3, \sigma\otimes \operatorname{id})\]
    is such that $\partial[E\otimes\mathbf{Q}_3,q_{1,\lambda},\alpha] = \delta \in W_{\Gamma}(\mathbf{F}_3)$. Note that the symmetric $\mathbf{Q}_3[\Gamma]$-bilinear form $V$ representing $[E\otimes\mathbf{Q}_3,q_{1,\lambda},\alpha]$ is bounded. According to \cite[Theorem~B]{BT}, this implies that the form $V$ contains a stable almost unimodular lattice $M_3$. 
    By definition, the lattice $M_3$ is therefore even, 3-elementary and it has an isometry $\varphi$ with characteristic polynomial $S$. Its discriminant group $A(M_3)$ is an $\mathbf{F}_3$-vector space of odd dimension by Proposition \ref{residue}(b).
    Up to taking overlattices in $V$, we may assume that $M_3$ is maximal (see Proposition~\ref{lem:stable even overlattices}); we have that $v_3(\operatorname{det}M_3) = 1$ and the determinant of its discriminant form is $+1$ (by definition of $\delta$). Note that $M_3$ having an isometry with characteristic polynomial $S$ implies that its determinant is equal to $-3(\mathbf{Z}_3^\times)^2 = \operatorname{det}(L\otimes\mathbf{Z}_3)$ (see (\ref{eq:det formula})). Hence $M_3\cong L\otimes\mathbf{Z}_3$ and we can conclude.
\qedhere
\end{proof} 

\begin{proof}[Proof of  Theorem \ref{thm:salem for og10}]
The necessary part of the statement is Lemma~\ref{thm:necessary part og10}. For the sufficient part, we proceed as in the proof of \cite[Theorem 12.1]{B 25}.

Since $S$ is irreducible and $L$ is unique in its genus (see Remark~\ref{rem:unique up to iso - og10}), we can apply Proposition~\ref{thm:local to global irreducible polynomials} together with Proposition~\ref{thm:sufficient part og10 - local existence} to conclude that $L$ has an isometry with characteristic polynomial $S$.\qedhere
\end{proof}

In Example~\ref{ex:og10} we show that the set of Salem polynomials satisfying the assumptions from Theorem~\ref{thm:salem for og10} is not empty.

\begin{example}
\label{ex:og10}
    Let $S(x) := x^{24}-x^{23}-x^{22}+x^{18}-x^{16}+x^{12}-x^{8}+x^{6}-x^{2}-x+1$: it is a Salem polynomial with Salem number $\lambda \approx 1.568703475...$. It satisfies $S(1) = -1$ and $S(-1) = 3$. According to Theorem \ref{thm:salem for og10},
    the lattice $\Lambda_{\rm OG10}$ admits an isometry $f$ with characteristic polynomial $S$, and $f$ acts nontrivially on the discriminant group of $\Lambda_{\rm OG10}$. 
    We can check this computationally in the following way.
    \smallskip

    \noindent
    Let again $E := \mathbf{Q}[x]/(S)$, let ${E_0} := \mathbf{Q}(\lambda+\lambda^{-1})$ be its maximal real subfield and let $\mathcal{O}_E$ be a maximal order of $E$. In our situation we have that the equation order $\mathbf{Z}[\lambda]$ is already maximal (which is not guaranteed to be true in general, cf.\ Example~\ref{ex:og6}),
    so we can choose $\mathcal{O}_E = \mathbf{Z}[\lambda]$. 
    Let $M$ be a finitely generated free $\mathbf{Z}$-module of rank divisible by 24. By transfer along $\operatorname{Tr}^E_{\mathbf{Q}}$, equipping $M$ with an even nondegenerate symmetric bilinear form $q$ and an isometry $g\in O(M, q)$ with minimal polynomial $S$ is therefore the same as defining a projective $\mathbf{Z}[\lambda]$-module structure on $M$ and equipping it with a (nondegenerate) hermitian form $h$ whose norm satisfies $\mathfrak{n}(M, h)\subset \mathcal{D}_{{E_0}/\mathbf{Q}}^{-1}$ (inverse different). Suppose that $M$ is a projective $\mathbf{Z}[\lambda]$-module of rank 1 and let $h$ be a nondegenerate hermitian form on $M$. We write $M$ as $\mathfrak{A}x$ (pseudobasis) where $x\in M\otimes E$ and $\mathfrak{A}$ is a fractional ideal of $\mathbf{Z}[\lambda]$.  The associated trace form $(M, \operatorname{Tr}^E_{\mathbf{Q}}\circ h)$ is isometric to $\Lambda_{\rm OG10}$ if and only if
    \begin{enumerate}
        \item $\mathfrak{n}(M, h) = h(x)\mathfrak{A}\overline{\mathfrak{A}}\subset \mathcal{D}_{{E_0}/\mathbf{Q}}^{-1}$,
        \item $\operatorname{N}^E_{\mathbf{Q}}(\mathcal{D}_{E/\mathbf{Q}}\cdot \mathfrak{n}(M, h)) = 3\mathbf{Z}$,
        \item there exists a unique real place $\nu$ of ${E_0}$ extending to conjugate complex places of $E$ such that $\nu(h(x)) > 0$,
        \item there exists a unique real place $\mu$ of ${E_0}$ extending to a real place of $E$ such that $\mu(h(x)) > 0$.
    \end{enumerate}
    Using algorithms designed by Kirschmer \cite{kir16} it is effectively possible to enumerate a transversal for the (finitely many) isometry classes of hermitian lattices satisfying the above properties. Computing their trace lattices, we obtain isometries of $\Lambda_{\rm OG10}$ with characteristic polynomial $S$.
\end{example}

\section{Isometries of lattices for type {\rm OG6}}

Throughout this section, we let $L := \Lambda_{\rm OG6} := U^{\oplus3}\oplus A_1^{\oplus2}$, and we let $S$ be a Salem polynomial of degree 8. We prove the following.

\begin{theo}\label{thm:salem for og6} 
The lattice $L$ has an isometry with characteristic polynomial $S$ if and only if 
$|S(1)|$ and $|S(-1)|$ are both squares
or both twice a square.
\end{theo}

As in the previous section, we start by showing the necessary part of Theorem~\ref{thm:salem for og6}.

\begin{lemma}\label{thm:necessary part og6}
    If the lattice $L$ has an isometry with characteristic polynomial $S$, then $|S(1)|$ and $|S(-1)|$ are both squares or both twice a square.
\end{lemma}

\begin{proof}
We proceed similarly as in the proof of Lemma~\ref{thm:necessary part og10}.
Denote the isometry by $\varphi$ and assume first that it induces a nontrivial action on the discriminant group $A(L)$ of $L$.
Consider an order 4 automorphism $\varphi'$ of $A_1^{\oplus2}(-1)$.
Then $(\varphi,\varphi')$ extends to an automorphism $\Phi\in O(\Lambda)$ 
for the even unimodular index 4 overlattice $\Lambda$ of $L\oplus A_1^{\oplus2}(-1)$.
By construction, $\Phi$ has characteristic polynomial $S(x)(x^2+1)$,
so \cite[Theorem 6.1]{GM} applies to show that
\[
|2S(1)|, \;| 2S(-1)| \textnormal{ and } {-S(1)S(-1)}\]
are all squares, and we recover the second alternative of the lemma.

If $\varphi$ induces a trivial action on $A(L)$,
then the same gluing argument applies separately to $\id$ and to $-\id$ acting on $A_1^{\oplus2}(-1)$.
With  linear factors in the resulting characteristic polynomial of $\Phi$, \cite[Theorem~1(a)]{B 21} implies that
\[
|S(-1) | \; \text{ is a square for $\varphi'=\id$, and } \; |S(1) | \; \text{ is a square for } \;
 \varphi'=-\id.
\]
Together this exactly gives the first alternative of the lemma.\qedhere

\end{proof}

\medskip

To prove the sufficient part of Theorem~\ref{thm:salem for og6}, we want to proceed as in the proof of Proposition~\ref{thm:sufficient part og10 - local existence}. However, since the determinant of $L$ is even, proving the existence of an isometry of $L\otimes\mathbf{Z}_2$ with characteristic polynomial $S$ is nontrivial. One main obstruction is the application of Proposition~\ref{residue} which offers little room to control the parity of integral $\mathbf{Z}_2$-lattices. 

\begin{remark}\label{rem:strategy}
    We comment on two of the main obstacles to the generalization of the proof of Proposition~\ref{thm:sufficient part og10 - local existence} to the proof of Theorem~\ref{thm:salem for og6}. 
    
    The Witt group of torsion $\mathbf{Z}_p$-bilinear forms (see \cite[Section~4.2]{BT}) in the case $p=2$ is not as practical as for when $p$ is odd. One of the main reasons is that integral $\mathbf{Z}_2$-lattices need not be even. One way to solve such an issue would probably be to study the Witt group of \emph{torsion quadratic forms} on finite abelian 2-groups and generalize some of the results of \cite[Sections 3 and 4]{BT}.

    Note that in \cite[Section 8]{BT} the authors prove results on the existence of isometries of even unimodular $\mathbf{Z}_2$-lattices. As they recall in \cite[Proposition~8.3]{BT}, for a given (even) rank and up to isometry, there are only two even unimodular $\mathbf{Z}_2$-lattices. This is not the case anymore for even \emph{almost unimodular} $\mathbf{Z}_2$-lattices. For instance, including the one of $L\otimes \mathbf{Z}_2$, there are 10 isometry classes of even almost unimodular $\mathbf{Z}_2$-lattices of rank 8 and whose determinant has 2-valuation 2. A more careful study of such almost unimodular $\mathbf{Z}_2$-lattices, their isometries and their associated class in the Witt group of torsion quadratic forms from the previous paragraph could help to recover results similar to \cite[Section 8]{BT}.
\end{remark}

In what follows we therefore choose a more direct, computational approach. The proofs of Lemmas~\ref{thm:sufficient og6 prime 2 - twice a square} and \ref{thm:sufficient og6 prime 2 - perfect square} use ideas which are actually very similar to what is described in Remark~\ref{rem:strategy}. Our strategy is the following.

Up to isometry, there are two even $\mathbf{Z}_2$-lattices of rank 8 with the same discriminant group as $L\otimes\mathbf{Z}_2$, as torsion quadratic forms. These two isometry classes are distinguished by their determinant class being either $-4(\mathbf{Z}_2^\times)^2$ for the isometry class of $L\otimes\mathbf{Z}_2$, and $12(\mathbf{Z}_2^\times)^2$ for the other one \cite[Theorem~1.9.1]{N}. We can use the determinant formula (\ref{eq:det formula}) to deduce that between the previous two isometry classes, only the one of $L\otimes\mathbf{Z}_2$ can have an isometry with characteristic polynomial $S$ if $-S(1)S(-1)$ is a square. For our goal, it is therefore enough to show that if $|S(1)|$ and $|S(-1)|$ are both squares or both twice a square, we can find an even $\mathbf{Z}_2$-lattice of rank 8, whose discriminant group is isometric to the one of $L\otimes\mathbf{Z}_2$, and which has an isometry with characteristic polynomial $S$.

\begin{lemma}\label{thm:sufficient og6 prime 2 - twice a square}
    Assume that $|S(1)|$ and $|S(-1)|$ are both twice a square. Then the lattice $L\otimes\mathbf{Z}_2$ admits an isometry with characteristic polynomial $S$.
\end{lemma}

\begin{proof}
    According to \cite[Theorem~1(a)]{B 21}, there exists an even unimodular $\mathbf{Z}_2$-lattice $M$ which admits an isometry $\Phi$ with characteristic polynomial $S(x)(x^2+1)$. For the rest of the proof, we implicitly see $\Phi$ as an isometry of $M\otimes\mathbf{Q}_2$. Let $I := \ker(\Phi^2+\id)\cap M$: it is an even $\mathbf{Z}_2$-lattice of rank 2 admitting an isometry with characteristic polynomial $x^2+1$. Up to rescaling the form on $M$ by $-1$ if necessary, there exists a basis $\{e,f\}$ of $I$ whose Gram matrix is
    \[\begin{pmatrix}
        2^{m}\alpha&0\\0&2^{m}\alpha
    \end{pmatrix},\]
    where $m\geq 1$ is a positive integer, $\alpha\in\{1,5\}$ and the action of $\Phi$ on said basis is given by the companion matrix of $x^2+1$. We distinguish two cases.\smallskip
    
    If $m = 2m'$ is even, then the $\mathbf{Z}_2$-lattice $\mathbf{Z}_2\frac{e+f}{2^{m'}}\oplus \mathbf{Z}_2\frac{f-e}{2^{m'}}$ is a $\Phi$-stable even overlattice of $I$. For the basis given above, its Gram matrix is
    $\begin{pmatrix}
        2\alpha&0\\0&2\alpha
    \end{pmatrix}$.\smallskip
    
    If $m = 2m'+1$ is odd, we conclude the same for the $\Phi$-stable even overlattice $\mathbf{Z}_2\frac{e}{2^{m'}}\oplus \mathbf{Z}_2\frac{f}{2^{m'}}$ of $I$.\smallskip
    
    In both of the cases, we conclude that $I$ admits an even overlattice with discriminant group $\begin{pmatrix}
        1/2&0\\0&1/2
    \end{pmatrix}\mod2\mathbf{Z}_2$ and which has an isometry with characteristic polynomial $x^2+1$. Because $M$ is unimodular, the discriminant groups of $I$ and $N:= \ker S(\Phi)\cap M$ are anti-isometric and the induced action of $\Phi$ on both discriminant groups coincide. This implies that our previous observation also holds for $N$, but for the rescaled form on discriminant group (see Lemma~\ref{lem:stable even overlattices}). In summary, the even $\mathbf{Z}_2$-lattice $N$ of rank 8 has an even overlattice with discriminant group 
    $\begin{pmatrix}
        -1/2&0\\0&-1/2
    \end{pmatrix}\mod2\mathbf{Z}_2$ and which has an isometry with characteristic polynomial $S$. We conclude by observing that the above torsion quadratic form is isometric to the discriminant group of $L\otimes \mathbf{Z}_2$.
    \qedhere
\end{proof}

\medskip

\begin{lemma}\label{thm:sufficient og6 prime 2 - perfect square}
    Assume that $|S(1)|$ and $|S(-1)|$ are both squares. Then the lattice $L\otimes\mathbf{Z}_2$ admits an isometry with characteristic polynomial $S$.
\end{lemma}

\begin{proof}
    We proceed similarly as in the proof of Lemma~\ref{thm:sufficient og6 prime 2 - twice a square}. As a first remark, we should note that since $S$ has degree 8 and both $|S(1)|$ and $|S(-1)|$ are squares, these numbers are both  squares of nonzero even integers. In particular, the polynomial $S$ is ramified at 2. Thus we know from \cite[Theorem~2]{B 21} that the even unimodular $\mathbf{Z}$-lattice $M := U^{\oplus5}$ has a semisimple isometry $\Phi$ with characteristic polynomial $S(x)(x^2-1)$. 
    Let $I:=\ker(\Phi^2-\id)$ be even of rank 2. Since $I$ has an involution with characteristic polynomial $x^2-1$, there exist two orthogonal primitive vectors $e,f\in I$ such that $\Phi(e) = e$ and $\Phi(f) = -f$ hold; by construction, either $I= \mathbf{Z}e\oplus \mathbf{Z}f$ or $I = \mathbf{Z}e\oplus \mathbf{Z}\frac{e+f}{2}$. Moreover, since $N:=\ker S(\Phi)$ has an isometry of characteristic polynomial $S$, its signature $(n_+, n_-)$ consists of odd integers. Since the same holds for $M$, we deduce that $I$ is definite, and thus $e^2 = 2k$ and $f^2=2l$ have the same sign. Because $M$ is unimodular, we can use the determinant formula (\ref{eq:det formula}) on $N$, which has an isometry with characteristic polynomial $S$, to infer that $kl$ must also be a square. So we distinguish again 2 cases.

    If $v_2(k)$ and $v_2(l)$ are both even, we write $v_2(k) = 2n$ and $v_2(l) = 2m$ with $n$ and $m$ nonnegative, and we observe that 
    \[\mathbf{Z}_2\frac{e}{2^n}\oplus \mathbf{Z}_2\frac{f}{2^{m}},\]
    seen as a $\mathbf{Z}_2$-lattice in $M\otimes\mathbf{Q}_2$, is a  $(\Phi\otimes\id)$-stable even almost unimodular overlattice of $I\otimes\mathbf{Z}_2$.

    Up to rescaling the form on $M$ by $-1$ if necessary, the even $\mathbf{Z}_2$-lattice $I\otimes \mathbf{Z}_2$ has a $(\Phi\otimes\id)$-stable even overlattice with Gram matrix $\begin{pmatrix}
        2\alpha&0\\0&2\alpha
    \end{pmatrix}$, where $\alpha\in\{1,5\}$. As we have concluded in the proof of Lemma~\ref{thm:sufficient og6 prime 2 - twice a square}, we can apply the same strategy to $N\otimes\mathbf{Z}_2$ to conclude that there exists an even almost unimodular $\mathbf{Z}_2$-lattice of rank 8 and discriminant group  $\begin{pmatrix}
        -1/2&0\\0&-1/2
    \end{pmatrix}\mod2\mathbf{Z}_2$, and which admits an isometry with characteristic polynomial $S$.
    \smallskip

    Now, if $v_2(k)$ and $v_2(l)$ are both odd, we write $v_2(k) = 2n+1$ and $v_2(l) = 2m+1$ with $n$ and $m$ again nonnegative, and we realize that
    \begin{equation}
    \label{eq:Z2-lattice}
        \mathbf{Z}_2\frac{2^me+2^nf}{2^{m+n+1}}\oplus\mathbf{Z}_2\frac{2^me-2^nf}{2^{m+n+1}}
    \end{equation}
    is a $(\Phi\otimes\id)$-stable even almost unimodular overlattice of $I\otimes\mathbf{Z}_2$. Note that the action of $\Phi$ on the discriminant group of the  $\mathbf{Z}_2$-lattice \eqref{eq:Z2-lattice} is given by exchanging the two given generators. This leads to a contradiction. In fact, if we mirror this process to the lattice $N\otimes \mathbf{Z}_2$, we obtain an isometry of $L\otimes\mathbf{Z}_2$ with characteristic polynomial $S$ and nontrivial action on the discriminant group. However, by the proof of Lemma~\ref{thm:necessary part og6}, this implies that $|S(1)|$ is twice a square, contradicting our assumptions on $S$. Hence this second case cannot occur.\qedhere
    
\end{proof}

\medskip

\begin{prop}\label{thm:sufficient part og6 - local existence}
Assume that $|S(1)|$ and $|S(-1)|$ are both squares, or both twice a square. Then for every prime number $p$, the lattice $L\otimes\mathbf{Z}_p$ admits an isometry with characteristic polynomial $S$.
\end{prop}

\begin{proof}
    If $p$ is odd, then 2 is invertible in $\mathbf{Z}_p$ and our assumptions give that $v_p(S(1))$ and $v_p(S(-1))$ are both even. Hence, as in the proof of Proposition~\ref{thm:sufficient part og10 - local existence},  the result follows from applying Proposition~\ref{p odd}. The case $p=2$ follows from Lemmas~\ref{thm:sufficient og6 prime 2 - twice a square} and \ref{thm:sufficient og6 prime 2 - perfect square}.
    \qedhere
\end{proof}

\medskip
\begin{proof}[Proof of  Theorem \ref{thm:salem for og6}]
The necessary part of the statement is Lemma~\ref{thm:necessary part og6}.

For the sufficient part, we remark that since $L$ is even indefinite of rank 8 and its discriminant group can be generated by two elements, the same result of Nikulin as the one used in  Remark~\ref{rem:unique up to iso - og10} applies and we can deduce that $L$ is unique in its genus, up to isometry. Knowing this, and the fact that $S$ is irreducible, we apply again Proposition~\ref{thm:local to global irreducible polynomials} together with Proposition~\ref{thm:sufficient part og6 - local existence} to conclude that $L$ has an isometry with characteristic polynomial $S$.
\end{proof}

As in the OG10 case, we give in Examples~\ref{ex:og6} and \ref{ex:og6-2} Salem polynomials satisfying the conditions of Lemmas~\ref{thm:sufficient og6 prime 2 - perfect square} and \ref{thm:sufficient og6 prime 2 - twice a square}, respectively.

\begin{example}
\label{ex:og6}
        Let $S(x) = x^{8}-x^{7}-x^{6}-x^{5}-x^{3}-x^{2}-x+1$: it is a Salem polynomial with Salem number $\lambda\approx 1.9164982637...$. It satisfies $S(1) = -4$ and $S(-1) = 4$. According to Theorem \ref{thm:salem for og6} ,
         the  lattice $\Lambda_{\rm OG6}$ admits an isometry $\varphi$ with characteristic polynomial $S$.
         This can be verified as in Example \ref{ex:og10}, but the computations are  more involved because the equation order of $\mathbf{Q}[x]/(S)$ is not maximal, in this case.
\end{example}

\begin{example}\label{ex:og6-2}
     Let $S(x) = x^{8}-x^6-x^5-x^3-x^2+1$: it is a Salem polynomial with Salem number $\lambda\approx 1.4579874796...$. It satisfies $S(1) = -2$ and $S(-1) = 2$. Moreover, the  lattice $\Lambda_{\rm OG6}$ admits an isometry $\varphi$ with characteristic polynomial $S$ and in this case, $\varphi$ acts nontrivially on the discriminant group of $\Lambda_{\rm OG6}$.
\end{example}

\section{Period map and Hodge Torelli}

For a fixed deformation type of hyperk\"ahler manifolds with reference lattice $\Lambda$, there is a coarse moduli space $\mathcal{M}_\Lambda$ parameterizing equivalence classes of pairs $(X, \eta)$ where $X$ is hyperk\"ahler of the given deformation type, and $\eta\colon (H^2(X, \mathbf{Z}), q_X)\to \Lambda$ is an isometry, called a $\Lambda$-marking. 
The pair $(X, \eta)$ is referred to as a \emph{$\Lambda$-marked pair}.
Two such $\Lambda$-marked pairs $(X, \eta)$ and $(X', \eta')$ are equivalent, in the sense of the definition of $\mathcal{M}_{\Lambda}$, if there is a biholomorphism $f\colon X\to X'$ such that $\eta' = \eta\circ f^\ast$. The moduli space $\mathcal{M}_\Lambda$ has a period map
\[\mathcal{P}_\Lambda\colon \mathcal{M}_{\Lambda}\to \left\{\mathbf{C}\omega\in\mathbb{P}(\Lambda\otimes\mathbf{C})\mid \omega^2 = 0,\, (\omega, \overline{\omega}) > 0\right\},\; [X, \eta]\mapsto \eta(H^{2,0}(X)).\]

The following result is proved in \cite[Theorem 8.1]{Huy}.
\begin{prop}\label{thm:surj period map}
    Let $\mathcal{M}_{\Lambda}^\circ\subset\mathcal{M}_{\Lambda}$ be any connected component. Then the restriction of $\mathcal{P}_{\Lambda}$ to $\mathcal{M}_{\Lambda}^\circ$ is surjective.
\end{prop}

The space $\mathcal{M}_\Lambda$ is not connected, in general. According to \cite[Theorem~4.6]{Huy}, two points $[(X, \eta)],\, [(X', \eta')]\in \mathcal{M}_{\Lambda}$ lie in the same connected component if and only if $\eta^{-1}\circ\eta'$ is a parallel transport isometry (see \cite[Definition~1.1]{markman} for a definition). For a fixed $\Lambda$-marked pair $(X, \eta)$, the subgroup of isometries $\mu\in O(H^2(X, \mathbf{Z}), q_X)$ such that $[(X, \eta)]$ and $[(X, \eta\circ \mu)]$ lie in a same connected component is the so-called \emph{monodromy group} $\operatorname{Mon}^2(X)$ of $X$. It is of finite index in $O(H^2(X, \mathbf{Z}), q_X)$ (see \cite[Lemma 7.5]{markman}) and it is a normal subgroup for the known examples of hyperk\"ahler manifolds (cf. Table~\ref{tab:monodromy hk}). 

\medskip

Given a $\mathbf{Z}$-lattice $L$ of signature $(3,\, \operatorname{rk}L-3)$, the subgroup of $O(L)$ preserving the orientation of any positive definite three-space in $L\otimes\mathbf{R}$ is the so-called group of \emph{orientation-preserving} isometries, usually denoted by $O^+(L)$ \cite[Section~4]{markman}. It coincides with the kernel of the group homomorphism
\[\det\cdot \operatorname{spin}_{\mathbf{R}}\colon O(L)\to \mathbf{R}^\times/(\mathbf{R}^\times)^2\]
where $\det\colon  O(L)\to \{\pm 1\}$ is the determinant morphism and $\operatorname{spin}_{\mathbf{R}}$ is the real spinor norm (see \cite[Lemma~4.1]{markman} or \cite[Lemma~2.2]{bmm26}). Note that given a hyperk\"ahler manifold $X$, there is an inclusion $\operatorname{Mon}^2(X)\subset O^+(H^2(X,\mathbf{Z}), q_X)$ \cite[Lemma~4.1(2)]{markman}.

\begin{remark}\label{rem:very special monodromy}
    Looking at Table~\ref{tab:monodromy hk}, we remark that for a hyperk\"ahler manifold $X$ of the known deformation types, if $\varphi\in O(H^2(X,\mathbf{Z}), q_X)$ satisfies all of the following:
    \begin{enumerate}
        \item $\varphi$ is orientation-preserving,
        \item $\varphi$ acts trivially on the discriminant group of $(H^2(X,\mathbf{Z}), q_X)$,
        \item $\det(\varphi) = +1$,
    \end{enumerate}
    then $\varphi\in\operatorname{Mon}^2(X)$.
\end{remark}

Let $X$ be a hyperk\"ahler manifold. We denote by $\mathcal{C}_X$ the connected component of $\{x\in H^{1,1}(X, \mathbf{R})\mid q_X(x)>0\}$ containing a K\"ahler class, which we call the \emph{positive cone of $X$}. According to \cite[Section~4]{markman}, the stabilizer of $\mathcal{C}_X$ in $O(H^2(X, \mathbf{Z}), q_X)$ coincides with the group of orientation-preserving isometries $O^+(H^2(X, \mathbf{Z}), q_X)$. 

The positive cone $\mathcal{C}_X$ admits two wall-and-chamber decompositions prescribed by two sets $\mathcal{W}^{pex}(X)$ and $\mathcal{W}(X)$
of integral $(1,1)$-classes, satisfying $\mathcal{W}^{pex}(X)\subset \mathcal{W}(X)\subset \operatorname{NS}(X) = H^2(X,\mathbf{Z})\cap H^{1,1}(X)$. Their respective elements are known as \emph{(stably) prime exceptional divisors} and \emph{wall divisors}.

\begin{remark}
    According to \cite[Theorem~6.18]{markman}, the vectors in $\mathcal{W}^{pex}(X)$ are roots of $(H^2(X,\mathbf{Z}), q_X)$. We can thus define the \emph{Weyl group} $W(X)$ of $X$ to be the subgroup of $O(H^2(X,\mathbf{Z}), q_X)$ generated by the reflections in prime exceptional divisors $v\in \mathcal{W}^{pex}(X)$. Every element in $W(X)$ is a monodromy operator
    and it preserves the Hodge structure on $H^2(X, \mathbf{Z})$.
\end{remark}

 The connected component of 
\[\mathcal{C}_X\setminus \bigcup_{v\in \mathcal{W}^{pex}(X)}v^\perp\]
containing the K\"ahler cone $\mathcal{K}_X$ of $X$ is called the \emph{fundamental exceptional chamber} of $X$, and it is usually denoted by $\mathcal{FE}_X$. The K\"ahler cone of $X$ is itself a connected component of
\[\mathcal{C}_X\setminus \bigcup_{v\in \mathcal{W}(X)}v^\perp.\]

\begin{defn}
    Let $X$ be a hyperk\"ahler manifold. We define:
    \begin{enumerate}
        \item the group of \emph{Hodge monodromies} of $X$ to be the subgroup $\operatorname{Mon}^2_{\rm Hdg}(X)\subset \operatorname{Mon}^2(X)$ of monodromies preserving the Hodge structure on $H^2(X, \mathbf{Z})$,
        \item the group $\operatorname{Mon}^2_{\rm bir}(X)\subset \operatorname{Mon}^2_{\rm Hdg}(X)$ to be the subgroup of Hodge monodromies preserving the fundamental exceptional chamber of $X$,
        \item the group $\operatorname{Mon}^2_{\rm aut}(X)\subset \operatorname{Mon}^2_{\rm bir}(X)$ to be the subgroup of Hodge monodromies preserving the K\"ahler cone of $X$;
        \item 
        the group of bimeromorphisms $X\dashrightarrow X$ 
        by $\operatorname{Bir}(X)$ and
        \item 
        the subgroup of biholomorphisms
        by $\operatorname{Aut}(X)$.
    \end{enumerate}
\end{defn}

With these definitions, we can now state a Hodge-theoretic version of the global Torelli theorem for hyperk\"ahler manifolds, following \cite[Proposition~9.1]{Huy}, \cite[Theorem~1.3]{markman} and \cite[Theorem~2.1]{ht13}.

\begin{theo}[Hodge-theoretic Torelli theorem]
\label{thm:Hodge-Torelli}
    Let $X$ be a hyperk\"ahler manifold, and let $b_X\colon \operatorname{Bir}(X)\to O(H^2(X,\mathbf{Z}), q_X)$ be the natural orthogonal representation. Then all of the following hold:
    \begin{enumerate}
        \item the kernel of $b_X$ is finite and its isomorphism class depends only on the deformation type of $X$,
        \item the image of $b_X$ coincides with $\operatorname{Mon}^2_{\rm bir}(X)$,
        \item the image of $\operatorname{Aut}(X)$ under $b_X$ coincides with $\operatorname{Mon}^2_{\rm aut}(X)$.
    \end{enumerate}
\end{theo}

We report in Table~\ref{tab:monodromy hk} the description of the monodromy group $\operatorname{Mon}^2(X)$ and of the kernel of $b_X$ for $X$ a hyperk\"ahler manifold of the known deformation types and of dimension greater than 2. These were determined in \cite[Lemma 4.2]{mar10}, \cite[Theorem~4.3]{mon16}, \cite[Theorem~1.4]{mr21} and \cite[Theorem~5.4]{ono22} for the monodromy groups, and in \cite[Proposition~10]{bea83}, \cite[Corollary~5]{bnws11}, and \cite[Theorems~3.1 and 5.2]{mw17} for the kernels.

\begin{table}
    \caption{Monodromy and cohomologically trivial symmetries}
    \label{tab:monodromy hk}
\centering\setlength{\tabcolsep}{7pt}
\renewcommand\arraystretch{1.5}
\resizebox{\textwidth}{!}{
\begin{tabular}{ccc}
    
        deformation type of $X$&$\operatorname{Mon}^2(X)$&$\ker b_X$\\
        \hline
        
        $\mathrm{Kum}_n$&$\{\varphi\in O^+(H^2(X,\mathbf{Z}), q_X)\mid \operatorname{det}(\varphi)A(\varphi) = \id\}$&$C_{n+1}^4\rtimes C_2$\\
         \hline
        OG6&$O^+(H^2(X,\mathbf{Z}), q_X)$&$C_2^8$\\
        \hline
        $\mathrm{K3}^{[n]}$&$\{\varphi\in O^+(H^2(X,\mathbf{Z}), q_X)\mid A(\varphi) = \pm \id\}$&$1$\\
         \hline
        OG10&$O^+(H^2(X,\mathbf{Z}), q_X)$&$1$\\
        \hline
    \end{tabular}
    }
    
\end{table}

\section{Biholomorphisms of elliptic hyperkähler manifolds}
\label{s:iso_HK}

Given a hyperk\"ahler manifold $X$, we recall that its transcendental lattice $\operatorname{T}(X)$ is the smallest primitive sublattice of $(H^2(X,\mathbf{Z}), q_X)$ whose $\mathbf{C}$-span contains $H^{2,0}(X)$. When $X$ is hyperbolic or elliptic, it coincides with $\operatorname{NS}(X)^\perp$.

\begin{theo}\label{thm:Torelli NS trivial}
    Let $X$ be an elliptic hyperk\"ahler manifold, let $S$ be a Salem polynomial whose degree satisfies $4\leq \deg (S)\leq b_2(X)$ and let $\lambda$ be the corresponding Salem number. Suppose that there exists an isometry $\varphi\in \operatorname{Mon}^2(X)$ satisfying both of the following:
    \begin{enumerate}
        \item $\varphi\mid_{H^{2,0}(X)} = \zeta\cdot \id$ where $S(\zeta) = 0$ and $|\zeta|=1$,
        \item $\varphi\mid_{\operatorname{NS}(X)} = \id$.
    \end{enumerate}
    Then $X$ admits a biholomorphism of dynamical degree $\lambda$.
\end{theo}

\begin{proof}
    By assumption (1), the isometry $\varphi$ preserves the Hodge structure on $H^2(X,\mathbf{Z})$ and in particular $\varphi\in \operatorname{Mon}^2_{\rm Hdg}(X)$. By assumption (2), $\varphi$ acts trivially on $\operatorname{NS}(X)$, meaning in particular that $\varphi$ fixes all wall divisors on $X$ and thus that it preserves the K\"ahler cone $\mathcal{K}_X$ of $X$. To see this, we remark that the set $\mathcal{W}(X)\subset \operatorname{NS}(X)$ of wall divisors on $X$ decomposes as a disjoint union
    \[\mathcal{W}(X) = \mathcal{W}_+(X)\sqcup \mathcal{W}_-(X)\]
    in such a way that
    \begin{enumerate}[label=(\alph*)]
        \item $\mathcal{W}_-(X) = \{-v\;\mid\; v\in\mathcal{W}_+(X)\}$,
        \item $\mathcal{K}_X = \left\{w\in \mathcal{C}_X\;\mid\; q_X(v,w) > 0,\,\forall v\in\mathcal{W}_+(X)\right\}$,
    \end{enumerate}
    where, as before, $\mathcal{C}_X$ denotes the positive cone of $X$. Since $\varphi$ acts trivially on $\mathcal{W}_+(X)$, we directly see that $\varphi$ preserves $\mathcal{K}_X$ (the isometry $\varphi$ is said to be \emph{positive} in McMullen's terminology, see for instance \cite{Mc2}). Hence Theorem~\ref{thm:Hodge-Torelli} implies that there exists some biholomorphism $f$ of $X$ whose induced cohomological action coincides with $\varphi$. In particular, the dynamical degree of $f$ is exactly $\lambda$.
    \qedhere
\end{proof}

\begin{coro}\label{thm:Torelli simple}
    Let $X$ be an elliptic hyperk\"ahler manifold with $\operatorname{NS}(X) = \{0\}$, let $S$ be a Salem polynomial of degree $b_2(X)$ and let $\lambda$ be the corresponding Salem number. Then $X$ admits a biholomorphism of dynamical degree $\lambda$ if and only if there exists an isometry $\varphi\in \operatorname{Mon}^2(X)$ with $S(\varphi)=0$ and such that $\varphi\mid_{H^{2,0}(X)} = \zeta\cdot \id$, where $S(\zeta) = 0$ and $|\zeta|=1$.
\end{coro}

\begin{proof}
    The sufficient part of the statement follows directly from Theorem~\ref{thm:Torelli NS trivial}. For the necessary part, we remark that since $\operatorname{NS}(X)$ is trivial, the transcendental lattice $\operatorname{T}(X) = H^2(X, \mathbf{Z})$ of $X$ is of maximal rank. Hence any biholomorphism $f$ of $X$ 
    of dynamical degree $\lambda$ induces an isometry $f^\ast$ of $H^2(X,\mathbf{Z})$ with irreducible minimal polynomial $S$.
    Here any generator of $H^{2,0}(X)$ is an eigenvector 
    of $f^\ast_{\mathbf{C}}$ with eigenvalue $\zeta\in\mathbf{C}$ 
    where $\zeta$ is a root of $S$ satisfying $|\zeta|=1$.
    \qedhere
\end{proof}

\begin{remark}\label{rem:aut equal bir}
    We note that Corollary~\ref{thm:Torelli simple} only deals with biholomorphisms. The reason is that if $X$ is a hyperk\"ahler manifold with $\operatorname{NS}(X)$ trivial, then $\operatorname{Aut}(X) = \operatorname{Bir}(X)$. Indeed, in that case $X$ has no wall divisors and its K\"ahler cone coincides with its positive cone. In particular, any bimeromorphism of $X$ is biregular, since it induces an orientation-preserving isometry of $(H^2(X,\mathbf{Z}), q_X)$.
\end{remark}

\subsection{Salem numbers of maximal degree on OG6 and OG10}

Let $n\in \{8,24\}$ and define
\[\Lambda_n:= \begin{cases}
    \Lambda_{\rm OG6}&\text{if $n=8$}\\
    \Lambda_{\rm OG10}&\text{if $n=24$}
\end{cases}.\]
An isometry $\varphi\in O^+(\Lambda_n)$ is called \emph{effective} if there exists a hyperk\"ahler manifold $X$ of OG6-type (resp. OG10-type) and a marking $\eta\colon H^2(X,\mathbf{Z})\to \Lambda_n$ such that $\varphi$ lies in $\eta\operatorname{Mon}^2_{\rm bir}(X)\eta^{-1}$. By definition of $\operatorname{Mon}^2_{\rm bir}(X)$, the previous means exactly that there exists some bimeromorphism $f$ of $X$ such that $f^\ast\in O(H^2(X,\mathbf{Z}), q_X)$ coincides with $\eta^{-1}\varphi\eta$. Since characteristic polynomials are invariant under similarity, we have that the characteristic polynomials of $f^\ast$ and $\varphi$ are the same. If moreover $\varphi\in\eta\operatorname{Mon}^2_{\rm aut}(X)\eta^{-1}$, i.e. we can choose $f$ to be biregular, then we say that $\varphi$ is \emph{regular effective}.

\medskip

Our strategy to prove Theorems~\ref{thm1} and \ref{thm2} builds on the existence of isometries of $\Lambda_n$ with Salem characteristic polynomial
constructed in the preceding sections. 
It remains to prove that such an isometry is regular effective. This is done in two steps, combining
Corollary~\ref{thm:Torelli simple} and the surjectivity of the period map for hyperk\"ahler manifolds (see Proposition~\ref{thm:surj period map}).

\medskip

We recall that if $X$ is of OG6-type or of OG10-type, then $\operatorname{Mon}^2(X) = O^+(H^2(X,\mathbf{Z}), q_X)$ is maximal (see Table~\ref{tab:monodromy hk}).

\begin{prop}\label{thm:effectivity salem isom}
    Let $\varphi\in O(\Lambda_n)$ be an isometry whose characteristic polynomial $S$ is a Salem polynomial of degree $n$.
    Then $\varphi$ lies in $O^+(\Lambda_n)$ and it is regular effective.
\end{prop}

\begin{proof} Since $S$ is monic and symmetric, we know that $\operatorname{det}\varphi = +1$. Moreover, by Zassenhaus's method to compute spinor norms (see for instance \cite[Equation~(2.1)]{Zas}), since $\Lambda_n$ has even rank and $S(-1)\neq 0$, the following holds:
\[\operatorname{spin}_{\mathbf{R}}(\varphi)=S(-1)(\mathbf{R}^\times)^2.\]
By the positivity of $S(-1)$ we conclude that $\varphi$ has positive real spinor norm. Hence $\varphi\in O^+(\Lambda_n)$. 
(This argument was inspired by the proof of \cite[Lemma 2.1]{kir19}.)

Because $\Lambda_n$ has signature $(3, n-3)$ and $S$ has only two real roots, there exists a non-real complex root $\zeta$ of $S$ satisfying that the real quadratic space $\ker(\varphi+\varphi^{-1}-(\zeta+\zeta^{-1})\id)$
is positive definite. Let $\mathbf{C}\omega\subset \Lambda_n\otimes\mathbf{C}$ be the eigenspace of $\varphi\otimes \id$ associated to $\zeta$. Since $\zeta$ is not real, we have that $\omega^2=0$ holds, and by the choice of $\zeta$, we therefore have that $(\omega,\overline{\omega})>0$ holds too. By the surjectivity of the period map for hyperk\"ahler manifolds (Proposition~\ref{thm:surj period map}), there exists a $\Lambda_n$-marked pair $(X, \eta)$ with period $\mathbf{C}\omega$, where $X$ is of OG6-type if $n=8$ and $X$ is of OG10-type if $n=24$. Since $\mathbf{C}\omega$ is preserved under the action of $\varphi$ and the characteristic polynomial of $\varphi$ is irreducible, we remark that 
\[\omega^\perp\cap \Lambda_n = \{0\}.\]
It follows that $\operatorname{NS}(X)$ is trivial. Transporting $\varphi$ to $H^2(X, \mathbf{Z})$ via $\eta$, we obtain an isometry $\eta^{-1}\varphi\eta\in O^+(H^2(X, \mathbf{Z}), q_X) = \operatorname{Mon}^2(X)$ whose characteristic polynomial is $S$ and of which $H^{2,0}(X)$ is an eigenspace (over $\mathbf{C}$) associated to 
some root $\zeta\in \mathbf{C}\setminus\mathbf{R}$ of $S$. Hence, according to Corollary~\ref{thm:Torelli simple}, we have $\eta^{-1}\varphi\eta\in \operatorname{Mon}^2_{\rm aut}(X)$ and $\varphi$ is regular effective, realized by some biholomorphism of $X$.
\qedhere
\end{proof}

\medskip
\begin{proof}[Proof of  Theorem \ref{thm1}] Follows from Theorem~\ref{thm:salem for og10} and Proposition~\ref{thm:effectivity salem isom}.
\end{proof}

\medskip
\begin{proof}[Proof of  Theorem \ref{thm2}] Follows from Theorem~\ref{thm:salem for og6} and Proposition~\ref{thm:effectivity salem isom}.
\end{proof}

\medskip
\begin{proof}[Proof of Corollary \ref{cor:char}]
According to the proof of \cite[Theorem~3.2]{Mc1}, which applies verbatim to hyperk\"ahler manifolds,
the characteristic polynomial of $f^\ast$ is a product of cyclotomic polynomials and of at most one Salem polynomial; 
this includes the notion of Salem polynomials of degree 2, which are all the polynomials of the form $x^2-bx+1$, where $b\geq 3$ is an integer. This implies in particular that the minimal polynomial $S$ of $f^\ast$ is either Salem or cyclotomic. By our assumption, the polynomial $S$ is not equal to $x-1$, and since $-\id\in O(H^2(X, \mathbf{Z}), q_X)$ is not orientation-preserving, we can conclude that $S$ is not linear and thus has even degree.

By the signature formula of McMullen (see \cite[Eq.~(4.4)]{Mc2}), if $S$ were cyclotomic,
then the signature of $(H^2(X, \mathbf{Z}), q_X)$ would consist of even integers. 
However, for all hyperk\"ahler manifolds, the positive index of inertia of the associated BBF-form is three. Thus $S$ cannot be cyclotomic, so it has to be a Salem polynomial. In that case, we also infer that $S$ is the characteristic polynomial of $f^\ast$. In particular, $\deg(S)=b_2(X)$, and since $b_2(X) \geq 3$, the polynomial $S$ is not quadratic.

To summarize, the characteristic polynomial of $f^\ast$ is equal to $S$, 
a Salem polynomial of even degree $b_2(X)\geq 3$.
The induced action of $f^\ast$ turns $H^2(X,\mathbf{Z})$ into a module of rank 1 over $\mathbf{Z}[x]/(S)$.
Hence the Hodge structure on $H^2(X,\mathbf{Z})$ is simple,
so $H^2(X,\mathbf{Z}) = \operatorname{T}(X)$ and thus $X$ is elliptic as claimed. 
The property that $f$ is biregular can now be deduced
from Remark~\ref{rem:aut equal bir}.

Among the known deformation types of hyperk\"ahler manifolds,
the evenness of $b_2(X)$ only leaves the three  types 
in the statement of Corollary \ref{cor:char}.
The given conditions then follow from \cite{B 25}, Theorem \ref{thm1} resp.\ Theorem \ref{thm2}.
\end{proof}

\subsection{Salem numbers realized from dimension 2}
\label{2toN}

Given $X$ a hyperk\"ahler manifold or a complex 2-torus, we call a bimeromorphism $f$ of $X$ \emph{algebraically trivial} if its induced action on $\operatorname{NS}(X)$ is trivial. In this concluding section, we show that the dynamical spectrum for elliptic hyperk\"ahler manifolds contains the dynamical degrees of algebraically trivial automorphisms of K3 surfaces and complex 2-tori.
The idea of the proof is rather simple.\smallskip

Start with a K3 surface (resp.\ complex 2-torus) $X$ together with an algebraically trivial biholomorphism $f$ of positive dynamical degree. In particular $X$ is elliptic. We use the well-known fact that $H^2(X,\mathbf{Z})$, equipped with the cup-product $\smile$, forms an even unimodular $\mathbf{Z}$-lattice which embeds into the BBF-form for OG10-type and $\mathrm{K3}^{[n]}$-type ($n\geq 2$) manifolds (resp. for hyperk\"ahler manifolds of all known types). Over $\mathbf{C}$ we can push the period $H^{2,0}(X)$ of $X$ along such an embedding and use the surjectivity of the period map for hyperk\"ahler manifolds (Proposition~\ref{thm:surj period map}) to conclude the existence of a hyperk\"ahler manifold $Y$ of the chosen deformation type whose transcendental lattice coincides with that of $X$ (after embedding). By extending the isometry induced by $f$ on $H^2(X,\mathbf{Z})$ to $H^2(Y,\mathbf{Z})$, requiring it to act as the identity on the orthogonal complement, we obtain an isometry in $O^+(H^2(Y,\mathbf{Z}), q_Y)$ which preserves the Hodge structure, acts trivially on $\operatorname{NS}(Y)$, has positive determinant and has trivial discriminant action. From that point, we can apply Theorem~\ref{thm:Torelli NS trivial}, to conclude that $Y$ has indeed an algebraically trivial biholomorphism of dynamical degree $\lambda$.

\begin{theo}\label{thm:k3 to hidger dim}
    Let $S$ be a Salem polynomial whose degree satisfies $4\leq \deg(S)\leq 22$ and let $\lambda$ be the associated Salem number. Suppose that there exists a K3 surface $X$ having an algebraically trivial biholomorphism $f$ of dynamical degree $\lambda$. 
    Then, both of the following hold:
    \begin{enumerate}
        \item there exists an elliptic OG10-type manifold having an algebraically trivial biholomorphism of dynamical degree $\lambda$,
        \item for every $n\geq 2$, there exists an elliptic $\mathrm{K3}^{[n]}$-type manifold having an algebraically trivial biholomorphism of dynamical degree $\lambda$.
    \end{enumerate}
\end{theo}

\begin{proof}
    Let $X$ and $f$ be as in the statement. By assumption, the isometry $\varphi\in O^+(H^2(X,\mathbf{Z}), \smile)$ induced by $f$ satisfies both of the following:
    \begin{itemize}
        \item $\varphi\mid_{\operatorname{T}(X)}$ has characteristic polynomial $S$,
        \item  $\ker(\varphi-\id) = \operatorname{NS}(X)$. 
    \end{itemize}
    Moreover, since $\varphi$ is of geometric origin, there exists a complex root $\zeta\in\mathbf{C}\setminus\mathbf{R}$ of $S$ such that $\varphi\mid_{H^{2,0}(X)} = \zeta\cdot \id$. Let $\Lambda := (H^2(X,\mathbf{Z}), \smile)\oplus A_2$: as a $\mathbf{Z}$-lattice, it is isometric to the BBF-form of an OG10-type manifold. We extend $\varphi$ to $\Phi\in O(\Lambda)$ as $\Phi:=\varphi\oplus\id_{A_2}$. Similarly to the beginning of the proof of Proposition~\ref{thm:effectivity salem isom}, we know that $\varphi$ has trivial determinant and trivial spinor norm. We thus observe all of the following:
    \begin{enumerate}[label=(\alph*)]
        \item $\operatorname{det}(\Phi) = +1$,
        \item $\Phi\in O^+(\Lambda)$,
        \item the minimal polynomial of $\Phi$ is $S(x)(x-1)$ with $\ker(\Phi-\id) = \operatorname{NS}(X)\oplus A_2$,
        \item in regard to the previous point, $\Phi$ acts trivially on the discriminant group of $\Lambda$.
    \end{enumerate}
    By the surjectivity of the period map for hyperk\"ahler manifolds (see Proposition~\ref{thm:surj period map}), there exists a $\Lambda$-marked pair $(Y, \eta)$ with $Y$ of OG10-type and with period $H^{2,0}(X)$ (seen as a subspace of $\Lambda\otimes\mathbf{C}$). We observe that 
    \[\eta(\operatorname{NS}(Y)) = (H^{2,0}(X)\oplus H^{0,2}(X))^\perp\cap\Lambda = \operatorname{NS}(X)\oplus A_2,\]
    meaning that $\eta(\operatorname{T}(Y)) = \operatorname{T}(X)$
    where we regard $\operatorname{T}(X)$
    as a sublattice of $\Lambda$.
    Moreover, the isometry $\Phi' := \eta^{-1}\Phi\eta$ lies in $\operatorname{Mon}^2(Y)$, by item (b) above (recall that the monodromy group of OG10-type manifolds is maximal; see Table~\ref{tab:monodromy hk}). Note that the isometry $\Phi'$ preserves the Hodge structure on $H^2(Y, \mathbf{Z})$ because $H^{2,0}(X)$ is an eigenspace of $\Phi$ for the eigenvalue $\zeta$. The isometry $\Phi'$ also acts trivially on $\operatorname{NS}(Y)$ by construction. Therefore, Theorem~\ref{thm:Torelli NS trivial} implies that $\Phi'$ is realized by some algebraically trivial biholomorphism of $Y$ of dynamical degree $\lambda$.

    The same proof as before applies verbatim for the $\mathrm{K3}^{[n]}$-types ($n\geq 2$), using that for every $n\geq 2$, the $\mathbf{Z}$-lattice $(H^2(X,\mathbf{Z}), \smile)\oplus A_1(n-1)$ is isometric to the BBF-form of a $\mathrm{K3}^{[n]}$-type manifold, and items (a)--(d) above ensure the monodromy conditions (cf. Remark~\ref{rem:very special monodromy}).
        \qedhere
\end{proof}

\medskip

\begin{theo}
\label{thm_from_tori}
    Let $S$ be a Salem polynomial whose degree satisfies $4\leq \deg(S)\leq 6$ and let $\lambda$ be the associated Salem number. Suppose that there exists a complex 2-torus $X$ having an algebraically trivial biholomorphism $f$ of dynamical degree $\lambda$. Then in each of the known deformation families of hyperk\"ahler manifolds, there exists an elliptic representative having an algebraically trivial biholomorphism of dynamical degree $\lambda$. 
\end{theo}

\begin{proof}
    Similar to the proof of Theorem~\ref{thm:k3 to hidger dim}, using that for a complex 2-torus $X$, the $\mathbf{Z}$-lattice $(H^2(X, \mathbf{Z}), \smile)$ is isometric to $U^{\oplus 3}$. The only point one has to be careful with is the monodromy condition. We refer to Table~\ref{tab:monodromy hk} for the description of the monodromy groups for the known hyperk\"ahler manifolds. We have highlighted items (a)--(d) in the proof of Theorem~\ref{thm:k3 to hidger dim} for this purpose (cf. Remark~\ref{rem:very special monodromy}). Further details are left to the reader.
        \qedhere
\end{proof}

\section*{Acknowledgments}

We thank Dominique Mattei for assisting us in the computation of the Salem polynomial in Example~\ref{ex:og10} and Simon Brandhorst for verifying the explicit realization in Example
~\ref{ex:og6}.

\section*{Tool and computational resource disclosure}
The Salem polynomials from Example~\ref{ex:og10}, and from Examples~\ref{ex:og6} and \ref{ex:og6-2} have been computed with the help of Claude AI and Gemini AI, respectively. The properties of the polynomials and the explicit realizations were subsequently verified by the authors, together with the computer algebra system OSCAR \cite{oscarbook,oscar}. 
ChatGPT was used for proofreading. Aside the previous statements, the authors declare that all the results of the paper have been obtained without the use of any AI.

\section*{Conflict of interest}
The authors declare no conflict of interest.

\end{document}